\documentstyle[amssymb,amsfonts]{amsart}

\def\N{{\hbox{\bf N}}}

 at 10 true pt

\newenvironment{proof}{\noindent {\bf Proof} }{\endprf\par}
\def \endprf{\hfill  {\vrule height6pt width6pt depth0pt}\medskip}
\def\emph#1{{\it #1}}
\def\textbf#1{{\bf #1}}

\theoremstyle{plain}
  \newtheorem{theorem}[subsection]{Theorem}

  \newtheorem{lemma}[subsection]{Lemma}

\theoremstyle{remark}

\theoremstyle{definition}
  \newtheorem{definition}[subsection]{Definition}

\include{psfig}

\begin{document}

\title[The prime ideals contain arbitrary large truncated classes]
{The prime ideals in every class contain arbitrary large
truncated classes}

\author{Chunlei Liu}
\address{Department of Mathematics, Shanghai Jiao Tong University, Shanghai, 200240}
\email{clliu@@sjtu.edu.cn}

\begin{abstract}We prove that the prime ideals in every class contain arbitrary large truncated classes.
\end{abstract}

\maketitle

\section{Introduction}
Green-tao \cite{gt-ap} proved the following epoch-making theorem.
\begin{theorem}[Green-Tao's PAP theorem]The primes contains arbitrary
long arithmetic progressions. \end{theorem}

We shall prove a generalization of
Green-Tao's PAP theorem to number fields.

Let $K$ be any number field. We embed it into its Minkowski space
$$K_{\infty}=\prod_{\sigma\mid\infty}K_{\sigma},$$
 where $K_{\sigma}$ is the completion
of $K$ at the archimedean place $\sigma$. The metric on $K_{\infty}$
is given by the formula
$$\|(x_{\sigma})\|_{\rm Min}^2=\sum_{\sigma\mid\infty}[K_{\sigma}:{\mathbb R}]\cdot\|x_{\sigma}\|^2.$$
So the balls
$$B_r=\{a\in K\mid \|a\|_{\rm Min}<r\},\ r>0$$
form a fundamental system of neighborhoods of $0$.

We view an arithmetic progression as a truncated residue class in
${\mathbb Z}$. The ideal-theoretic generalization of the notion of
residue classes in ${\mathbb Z}$ to number fields is the notion of
equivalence classes of ideals. Let $O_K$ be the ring of integers in
$K$.
\begin{definition}Let ${\frak m},{\frak a}, {\frak b}$ be nonzero ideals of $O_K$. If
there is a nonzero number $\xi\in1+{\frak m}{\frak a}^{-1}$ such
that
$${\frak b}=(\xi){\frak a},$$ then ${\frak b}$ is said to be
equivalent to ${\frak a}$ modulo ${\frak m}$.
\end{definition}
\begin{definition}Let ${\frak m}$ and ${\frak b}$ be nonzero fractional ideals of $O_K$ such that
 ${\frak m}\subseteq{\frak b}$. Let $a\in{\frak b}$ and $r>0$.
We call $$\{\xi\in {\frak b}\mid \xi\equiv a({\rm mod}{\frak m}),
\|\xi-a\|_{\rm Min}<r\}$$ a truncated residue class of ${\frak b}$
modulo ${\frak m}$. We call it a truncated principal residue class
of ${\frak b}$ if ${\frak m}$ is principal.
\end{definition}
\begin{definition}Let ${\frak m},{\frak b}$ be nonzero fractional ideals of $O_K$ such that
 ${\frak m}\subseteq{\frak b}$, and let $A$ be a truncated residue class of ${\frak b}$
modulo ${\frak m}$. We call $$\{\xi{\frak b}^{-1}\mid \xi\in A\}$$ a
truncated generalized class.
\end{definition}
We shall prove the following generalization of Green-tao's PAP
theorem.
\begin{theorem}\label{main}The prime ideals in every class contain arbitrary large
truncated generalized classes.\end{theorem}

 The proof of Theorem \ref{main} is
a generalization of the arguments of Green-Tao in \cite{gt-ap}. A
positive density version of the above theorem can be proved
similarly.
\section{Pseudo-random measures on inverse systems}
In this section we establish the relationship between two kinds of
measures on inverse systems.

Let ${\frak b}$ a fixed nonzero fractional ideal of $K$. For the
sake of convenience, we take ${\frak b}$ to be the inverse of a
nonzero integral ideal. Let $k$ be a fixed positive integer, and $I$
the set of positive integers which are prime to every nonzero number
in $O_K\cap B_{2k}$. Then $\{{\frak b}/(N{\frak b})\}_{N\in I}\}$ is
an inverse system of finite groups. For each $j\in O_K\cap B_k$, we
write $e_j=(O_K\cap B_k)\setminus\{j\}$. Then $(O_K\cap
B_k,\{e_j\}_{j\in O_K\cap B_k})$ is a hyper-graph. To each
hyper-edge $e_j$, we associate the inverse system $\{({\frak
b}/N{\frak b})^{e_j}\}_{N\in I}$. Thus the system $\{({\frak
b}/N{\frak b})^{e_j}\}_{N\in I,j\in O_K\cap B_k}$ maybe regarded as
an inverse system on the hyper-graph $(O_K\cap B_k,\{e_j\}_{j\in
O_K\cap B_k})$.

For each $j\in O_K\cap B_k$, and for each $N\in I$, let
$\tilde{\nu}_{N,j}$ be a nonnegative function on $({\frak b}/N{\frak
b})^{e_j}$.
\begin{definition}The system
$\{\tilde{\nu}_{N,j}\}_{N\in I,j\in O_K\cap B_k}$ is
called a \emph{pseudo-random} system of measures on the system $\{{\frak b}/(N{\frak b})\}_{N\in I, j\in O_K\cap B_k}$ if the following conditions are satisfied.\begin{enumerate}
                                    \item For all $j\in O_K\cap B_k$, and for all $\Omega_j\subseteq\{0,1\}^{e_j}\setminus\{0\}$, one has
$$ \frac{1}{N^{|e_j|[K:{\mathbb Q}]}}\sum_{x^{(1)} \in ({\frak b}/N{\frak b})^{e_j}}\prod_{\omega \in \Omega_j} \tilde{\nu}_{N,j}(x^{(\omega)})  = O(1),$$
uniformly for all $x^{(0)}\in ({\frak b}/N{\frak b})^{e_j}$
                                    \item Given any choice $\Omega_j\subseteq\{0,1\}^{e_j}$ for each $j\in O_K\cap B_k$, one has
$$
\frac{1}{N^{2|O_K\cap B_k|[K:{\mathbb Q}]}}\sum_{x^{(0)}, x^{(1)} \in({\frak b}/N{\frak b})^{O_K\cap B_k}}\prod_{j\in O_K\cap B_k} \prod_{\omega \in \Omega_j} \tilde{\nu}_{N,j}(x^{(\omega)})  = 1 + o(1),
$$as $ N \to \infty$ in $I$.
                                    \item For all $j\in O_K\cap B_k$, for all $i \in e_j$, for all $\Omega_j\subseteq \{0,1\}^{e_j}$, and for all $M \in{\mathbb N}$, we have
$$
\frac{1}{N^{2(|e_j|-1)[K:{\mathbb Q}]}}\sum_{x^{(0)}, x^{(1)} \in ({\frak b}/N{\frak b})^{e_j \backslash \{i\}} }(\frac{1}{N^{2[K:{\mathbb Q}]}}\sum_{ x_i^{(0)}, x_i^{(1)} \in{\frak b}/N{\frak b}}\prod_{\omega \in \Omega_j} \tilde{\nu}_{N,j}( x^{(\omega)} ))^M= O(1).
$$
                                  \end{enumerate}
\end{definition}
For each positive
integer $N\in I$, let $\tilde{\nu}_N$ be a nonnegative function on ${\frak b}/(N{\frak b})$. \begin{definition}The system
$\{\tilde{\nu}_N\}$ is said to satisfy the $k$-cross-correlation
condition if, given any positive integers $s\leq |O_K\cap B_k|2^{|O_K\cap B_k|},m\leq
2|O_K\cap B_k|$, and given any mutually independent linear forms
$\psi_1,\cdots,\psi_{s}$ in $m$ variables whose coefficients are
numbers in $O_K\cap B_{2k}$, we have
$$\frac{1}{N^{m[K:{\mathbb Q}]}}\sum_{\stackrel{x_i\in {\frak b}/(N{\frak b})}{i=1,\cdots,m}}
\prod_{j=1}^{s}\tilde{\nu}_N(\psi_j(x)+b_j)=1+o(1),\
N\rightarrow\infty$$ uniformly for all $b_1,\cdots,b_{s}\in {\frak
b}/(N{\frak b})$.
\end{definition}
\begin{definition}The system
$\{\tilde{\nu}_N\}$ is said to satisfy the $k$-auto-correlation
condition if, given any positive integers $s\leq |O_K\cap
B_k|2^{|O_K\cap B_k|}$, there exists a system $\{\tilde{\tau}_N\}$
of nonnegative functions  on $\{{\frak b}/(N{\frak b})\}$ which
obeys the moment condition
$$\frac{1}{N^{[K:{\mathbb Q}]}}\sum_{x\in{\frak b}/(N{\frak
b})}\tilde{\tau}^M(x)=O_{M,s}(1),\ \forall M\in{\mathbb N}$$ such
that
$$\frac{1}{N^{[K:{\mathbb Q}]}}
\sum\limits_{x\in {\frak b}/(N{\frak b})}\prod_{i=1}^s
 \tilde{\nu}_N(x +
y_i)\ll\sum_{1\leq i<j\leq
s}\tilde{\tau}(y_i-y_j).$$\end{definition}\begin{definition}The
system $\{\tilde{\nu}_N\}$ is called a $k$-pseudo-random system of
measure on the inverse system $\{{\frak b}/(N{\frak b})\}$ if it
satisfies
 the $k$-cross-correlation condition and the $k$-auto-correlation
condition.\end{definition}From now on we assume that
$$\tilde{\nu}_{N, j}(x) := \tilde{\nu}_N(\sum_{i\in e_j}(i-j)x_i).$$\begin{theorem}\label{inversesys}If
$\{\tilde{\nu}_N\}$ is $k$-pseudo-random, then
$\{\tilde{\nu}_{N,j}\}$ is pseudo-random.
\end{theorem}\begin{proof}First, we show that, for all $j\in O_K\cap B_k$, and for all $\Omega_j\subseteq\{0,1\}^{e_j}\setminus\{0\}$,
$$ \frac{1}{N^{|e_j|[K:{\mathbb Q}]}}\sum_{x^{(1)} \in ({\frak b}/N{\frak b})^{e_j}}\prod_{\omega \in \Omega_j} \tilde{\nu}_{N,j}(x^{(\omega)})  = O(1),$$ uniformly for all $x^{(0)}\in ({\frak b}/N{\frak b})^{e_j}$.
For each $\omega\in\Omega_j$, set
$$\psi_{\omega}(x^{(1)})=\sum_{i\in e_j,\omega_i=1}(i-j)x_i^{(1)},
$$
and $$b_{\omega}=\sum_{i \in e_j,\omega_i=0}(i-j)x_i^{(0)}).$$Then
$$ \frac{1}{N^{|e_j|[K:{\mathbb Q}]}}\sum_{x^{(1)} \in ({\frak b}/N{\frak b})^{e_j}}\prod_{\omega \in \Omega_j} \tilde{\nu}_{N,j}(x^{(\omega)})$$
$$ =\frac{1}{N^{|e_j|[K:{\mathbb Q}]}}\sum_{x^{(1)} \in ({\frak b}/N{\frak b})^{e_j}}\prod_{\omega \in \Omega_j} \tilde{\nu}_{N}(\psi_{\omega}(x^{(1)})+b_{\omega})  = O(1).$$

Secondly, we show that, given any choice $\Omega_j\subseteq\{0,1\}^{e_j}$ for each $j\in O_K\cap B_k$,
$$
\frac{1}{N^{2|O_K\cap B_k|[K:{\mathbb Q}]}}\sum_{x^{(0)}, x^{(1)} \in({\frak b}/N{\frak b})^{O_K\cap B_k}}\prod_{j\in O_K\cap B_k} \prod_{\omega \in \Omega_j} \tilde{\nu}_{N,j}(x^{(\omega)})  = 1 + o(1),
$$as $ N \to \infty$ in $I$.
For each pair $(j,\omega)$ with $j\in O_K\cap B_k$ and $\omega\in\Omega_j$, set
$$  \psi_{(j,\omega)}(x)=\sum_{\stackrel{i\in O_K\cap B_k,\delta=0,1}{\omega_i=\delta}}(i-j)x_i^{(\delta)}.
$$
Then $$
\frac{1}{N^{2|O_K\cap B_k|[K:{\mathbb Q}]}}\sum_{x^{(0)}, x^{(1)} \in({\frak b}/N{\frak b})^{O_K\cap B_k}}\prod_{j\in O_K\cap B_k} \prod_{\omega \in \Omega_j} \tilde{\nu}_{N,j}(x^{(\omega)}) $$$$ =
\frac{1}{N^{2|O_K\cap B_k|[K:{\mathbb Q}]}}\sum_{x^{(0)}, x^{(1)} \in({\frak b}/N{\frak b})^{O_K\cap B_k}}\prod_{j\in O_K\cap B_k} \prod_{\omega \in \Omega_j} \tilde{\nu}_{N}(\psi_{j,\omega}(x))  = 1 + o(1),
$$as $ N \to \infty$ in $I$.

Finally we show that,  for all $j\in O_K\cap B_k$, for all $i \in e_j$, for all $\Omega_j\subseteq \{0,1\}^{e_j}$, and for all $M \in{\mathbb N}$,
$$
\frac{1}{N^{2(|e_j|-1)[K:{\mathbb Q}]}}\sum_{x^{(0)}, x^{(1)} \in ({\frak b}/N{\frak b})^{e_j \backslash \{i\}} }(\frac{1}{N^{2[K:{\mathbb Q}]}}\sum_{ x_i^{(0)}, x_i^{(1)} \in{\frak b}/N{\frak b}}\prod_{\omega \in \Omega_j} \tilde{\nu}_{N,j}( x^{(\omega)} ))^M= O(1).
$$
By Cauchy-Schwartz it suffices to show that, for $a=0,1$, $$
\frac{1}{N^{2(|e_j|-1)[K:{\mathbb Q}]}}\sum_{x^{(0)}, x^{(1)} \in
({\frak b}/N{\frak b})^{e_j \backslash \{i\}} }
(\frac{1}{N^{[K:{\mathbb Q}]}}\sum_{ x_i^{(a)} \in{\frak b}/N{\frak
b}}\prod_{\omega \in \Omega_j,\omega_i=a} \tilde{\nu}_{N,j}(
x^{(\omega)} ))^{2M}= O(1).
$$
For each $\omega\in \Omega_j$ with $\omega_i=a$, set
$$\psi_{\omega}(x)=\sum_{l \in e_j\setminus\{i\}} (l-j)x_l^{(\omega_l)}.$$Then $$
\frac{1}{N^{2(|e_j|-1)[K:{\mathbb Q}]}} \sum_{x^{(0)}, x^{(1)} \in
({\frak b}/N{\frak b})^{e_j \backslash \{i\}} }
(\frac{1}{N^{[K:{\mathbb Q}]}}\sum_{ x_i^{(a)} \in{\frak b}/N{\frak
b}} \prod_{\omega \in \Omega_j,\omega_i=a} \tilde{\nu}_{N,j}(
x^{(\omega)} ))^{2M}$$
$$
\leq\frac{1}{N^{2(|e_j|-1)[K:{\mathbb Q}]}}\sum_{x^{(0)}, x^{(1)}
\in ({\frak b}/N{\frak b})^{e_j \backslash \{i\}}
}\sum_{\stackrel{\omega,\omega'\in\Omega_j}{\omega_i=\omega'_i=a}}\tilde{\tau}^{2M}(\psi_{\omega}(x)-\psi_{\omega'}(x))$$$$
=\sum_{\stackrel{\omega,\omega'\in\Omega_j}{\omega_i=\omega'_i=a}}\frac{1}{N^{[K:{\mathbb
Q}]}}\sum_{x \in {\frak b}/(N{\frak b}) }\tilde{\tau}^{2M}(x)=
O(1).$$
\end{proof}
\section{Pseudo-random measures on nonzero fractional ideals}
In this section we establish the relationship between measures on
inverse systems and measures on nonzero fractional ideals.

Let $A$ be a positive constant. For $N\in I$, let $\nu_N\ll\log^A N$
be a nonnegative function on ${\frak b}$.
\begin{definition}The system $\{\nu_N\}$ is said
to satisfy the $k$-cross-correlation condition if, given any
parallelotope $I$ in $K_{\infty}$, given any positive integers
$s\leq |O_K\cap B_k|2^{|O_K\cap B_k|},m\leq 2|O_K\cap B_k|$, given
any $N\log^{-2sA}N\leq \lambda<N$, and given any mutually
independent linear forms $\psi_1,\cdots,\psi_{s}$ in $m$ variables
whose coefficients are numbers in $O_K\cap B_{2k}$, we have
$$\frac{1}{|{\frak b}\cap(\lambda I)|^m}\sum_{\stackrel{x_i\in {\frak b}\cap(\lambda I)}{i=1,\cdots,m}}\prod_{j=1}^{s}\nu_N(\psi_j(x)+b_j)=1+o(1),\
N\rightarrow\infty$$ uniformly for all numbers $b_1,\cdots,b_{s}\in
{\frak b}$.
\end{definition}
\begin{definition}The system $\{\nu_N\}$
is said to satisfy the $k$-auto-correlation condition if given any
positive integers $s\leq |O_K\cap B_k|2^{|O_K\cap B_k|}$, there
exists a system $\{\tau_N\}$ of nonnegative functions  on ${\frak
b}$ such that, given any parallelotope $I$ in $K_{\infty}$,
$$\frac{1}{|(NI)\cap{\frak b}|}\sum_{x\in (NI)\cap{\frak b}}\tau_N^M(x)=O_{M}(1),\ \forall M\in{\mathbb N}$$
and
$$\frac{1}{|(NI)\cap{\frak b}|}
\sum\limits_{x\in (NI)\cap{\frak b}}\prod_{i=1}^s
 \nu_N(x +
y_i)\leq\sum_{1\leq i<j\leq
s}\tau_N(y_i-y_j).$$\end{definition}\begin{definition}The system
$\{\nu_N\}$ is $k$-pseudo-random if it satisfies
 the $k$-cross-correlation condition and the $k$-auto-correlation
condition.
\end{definition}
Let $\eta_1,\cdots,\eta_n$ be a ${\mathbb Z}$-basis of ${\frak b}$,
and set $$G=\sum_{j=1}^n(-1/2,1/2]\eta_i\subseteq K_{\infty}.$$ Let
$\varepsilon>0$ be a sufficiently small constant depending only on
$k$ and ${\frak b}$. From on on we assume that
$$\tilde{\nu}_{N}(x)=\left\{
                    \begin{array}{ll}\nu_N(\hat{x}), & \hbox{} x=\hat{x}+N{\frak b},\hat{x}\in\varepsilon NG,\\
                      1, & \hbox{otherwise.}
                    \end{array}
                  \right.
$$We now prove the following.
\begin{theorem}\label{pseudorandomrelation}If the system $\{\nu_N\}_{N\in I}$ is $k$-pseudo-random,
then the system $\{\tilde{\nu}_N\}_{N\in I}$ is also
$k$-pseudo-random.\end{theorem}
\begin{proof}
First we show that, given any positive integers $s\leq |O_K\cap B_k|2^{|O_K\cap B_k|},m\leq
2|O_K\cap B_k|$, and given any mutually independent linear forms
$\psi_1,\cdots,\psi_{s}$ in $m$ variables whose coefficients are
numbers in $O_K\cap B_{2k}$,
$$\frac{1}{N^{m[K:{\mathbb Q}]}}\sum_{\stackrel{x_i\in {\frak b}/(N{\frak b})}{i=1,\cdots,m}}
\prod_{j=1}^{s}\tilde{\nu}_N(\psi_j(x)+b_j)=1+o(1),\
N\rightarrow\infty$$ uniformly for all $b_1,\cdots,b_{s}\in {\frak
b}/(N{\frak b})$.
It suffices to show that for any $S'\subset \{1,\cdots,s\}$,
$$\frac{1}{N^{m[K:{\mathbb Q}]}}\sum_{\stackrel{x_i\in NG}{i=1,\cdots,m}}
\prod_{j\in S'}(\tilde{\nu}_N(\psi_j(x)+b_j)-1)=o(1),\
N\rightarrow\infty$$ uniformly for all $b_1,\cdots,b_{s}\in {\frak
b}$. Regard $\psi$ as an ${\mathbb R}$-linear map from
$K_{\infty}^m$ to $K_{\infty}^s$. There is a positive constant $c$
such that for any $x\in K_{\infty}^s$, the number of translations of
$G^s$ by vectors in $x+{\frak b}^s$ needed to cover $\psi(G^m)$ is
$\leq c$. Hence the number of translations of $NG^s$ by vectors in
$-b+N{\frak b}^s$ needed to cover $\psi(NG^m)$ is $O(1)$. Therefore it
suffices to show that, for any $\beta\in{\frak b}^S$,
$$ \frac{1}{N^{mn}}\sum\limits_{ \stackrel{x \in
(NG\cap {\frak b})^m}{\psi(x)\in -b+N\beta+NG^S}}\prod_{s \in S'}
(\tilde{\nu}_N(\psi_j(x)+b_j)-1)= o(1).$$ Let $Q=\log^{2A}N$. We
analyze the contributions to the left-hand side from the translates
of $(\frac{N}{Q}G)^m$ by vectors in $(NG\cap{\frak b})^m$. The
translations whose images under $\psi$ do not intersect with
$-b+N\beta+\varepsilon NG^S$ apparently make no contribution. The
total contributions from translations whose images under $\psi$ are
contained in $-b+N\beta+\varepsilon NG^S$ is equal to $$
\frac{1}{N^{mn}}\sum\limits_{ \psi(x_0+(\frac{N}{Q}G)^m)\subseteq
-b+N\beta+\varepsilon NG^S}\sum\limits_{x \in
x_0+(\frac{N}{Q}G)^m}\prod_{s \in S'}
(\nu_N(\psi_j(x)+b_j-N\beta)-1),$$ which is $o(1)$ by the
pseudo-randomness of $\{\nu_N\}$. It remains to consider the
contribution from translations whose images under $\psi$ intersect
with the boundary of $-b+N\beta+\varepsilon NG^S$. The total number
of such translations is bounded by $O(Q^{mn-1})$. As each such a
translation contributes at most $Q^{-mn}\log^{sA} N$. The total
contribution given by such translations is bounded by
$O(\frac{\log^{sA} N}{Q})$.

Secondly we show that,
 given any positive integers $s\leq |O_K\cap
B_k|2^{|O_K\cap B_k|}$,
$$\frac{1}{N^{[K:{\mathbb Q}]}}
\sum\limits_{x\in {\frak b}/(N{\frak b})}\prod_{i=1}^s
 \tilde{\nu}_N(x +
y_i)\ll\sum_{1\leq i<j\leq s}\tilde{\tau}(y_i-y_j),$$ where
$$\tilde{\tau}(x)=\tau(x),\ x\in NG.$$Set
$$g_N(x)=\left\{
                    \begin{array}{ll}
                      \nu_N(\hat{x}), & \hbox{} x=\hat{x}+N{\frak b},\hat{x}\in\varepsilon NG,\\
                      0, & \hbox{otherwise.}
                    \end{array}
                  \right.$$
Then $$ \frac{1}{N^{[K:{\mathbb Q}]}} \sum\limits_{ x \in {\frak
b}/(N{\frak b})}\prod_{i=1}^{s} \tilde{\nu}_N(x + y_i)$$
$$\leq \frac{1}{N^{[K:{\mathbb Q}]}} \sum\limits_{ x \in {\frak
b}/(N{\frak b})}\prod_{i=1}^{s}(1+g_N(x + y_i))$$$$\leq
\sum_{S'\subset \{1,\cdots,s\}}\frac{1}{N^{[K:{\mathbb Q}]}}
\sum\limits_{ x \in {\frak b}/(N{\frak b})}\prod_{i\in S'}g_N(x +
y_i).$$ So we are reduced to showing, for any $S'\subseteq
\{1,\cdots,s\}$, that $$\frac{1}{N^{[K:{\mathbb Q}]}} \sum\limits_{
x \in NG\cap{\frak b}}\prod_{i\in S'}g_N(x + y_i)\ll\sum_{i\neq j\in
S'}\tilde{\tau}(y_i-y_j), \forall y_i\in NG.$$ It suffices to show
that, for any $\beta\in{\frak b}^{S'}$, $$\frac{1}{N^{[K:{\mathbb
Q}]}} \sum\limits_{ x+y_i-N\beta_i \in NG\cap{\frak b},\forall i\in
S'}\prod_{i\in S'}g_N(x + y_i)\ll\sum_{i\neq j\in
S'}\tilde{\tau}(y_i-y_j), \forall y_i\in NG.$$ We
have
$$\frac{1}{N^{[K:{\mathbb Q}]}} \sum\limits_{ x+y_i-N\beta_i \in
NG\cap{\frak b},\forall i\in S'}\prod_{i\in S'}g_N(x +
y_i)$$$$=\frac{1}{N^{[K:{\mathbb Q}]}} \sum\limits_{ x+y_i-N\beta_i
\in \varepsilon NG\cap{\frak b},\forall i\in S'}\prod_{i\in
S'}\nu_N(x + y_i).$$ We may assume that
$y_i-N\beta_i-y_j+N\beta_j\in NG$. Then
$$\frac{1}{N^{[K:{\mathbb Q}]}} \sum\limits_{ x+y_i-N\beta_i \in
NG\cap{\frak b},\forall i\in S'}\prod_{i\in S'}g_N(x +
y_i)\ll\sum_{i\neq j\in S'}\tau(y_i-y_j)=\sum_{i\neq j\in
S'}\tilde{\tau}(y_i-y_j).$$The theorem is proved.
\end{proof}

\section{The relative Szemer\'edi theorem for number fields}
In this section we prove the relative Szemer\'edi theorem for number fields.

For each $N\in I$, let $\tilde{A}_N$ be a subset of
${\frak b}/(N{\frak b})$. \begin{definition}The upper density of $\{\tilde{A}_N\}$
relative to $\{\tilde{\nu}_N\}$ is defined to be
$$\limsup_{I\ni N\rightarrow\infty}
\frac{\sum_{x\in \tilde{A}_N}\tilde{\nu}_N(x)}{\sum_{x\in {\frak
b}/(N{\frak b})}\tilde{\nu}_N(x)}.$$\end{definition}
The following
version of the relative Szemer\'edi theorem for number fields
follows from a theorem of Tao in \cite{tao:gaussian}.
\begin{theorem} If the system $\{\tilde{\nu}_{N,j}\}$ is
pseudo-random, and $\{\tilde{A}_N\}$ has positive upper density
relative to $\{\tilde{\nu}_N\}$, then there is a subset
$\tilde{A}_N$ and a truncated  residue class of ${\frak b}$ of size
$|O_K\cap B_k|$ such that $$A({\rm mod} N{\frak b})\subseteq
\tilde{A}_N.$$
\end{theorem}The above theorem, along with Theorem \ref{inversesys}, implies the following.
\begin{theorem}\label{rsinversesys} If the system $\{\tilde{\nu}_{N}\}$ is
$k$-pseudo-random, and $\{\tilde{A}_N\}$ has positive upper density
relative to $\{\tilde{\nu}_N\}$, then there is a subset
$\tilde{A}_N$ and a truncated  residue class of ${\frak b}$ of size
$|O_K\cap B_k|$ such that
$$A({\rm mod} N{\frak b})\subseteq \tilde{A}_N.$$
\end{theorem}
\begin{definition}For $N\in I$,
let $A_N$ be a subset of ${\frak b}\cap B_N$. The upper density of
$\{A_N\}$ relative to $\{\nu_r\}$ is defined to be
$$\limsup_{N\rightarrow\infty}\frac{\sum_{x\in A_N}\nu_N(g)}{\sum_{x\in {\frak b}\cap B_N}\nu_N(x)}.$$\end{definition}
We now prove the following.
\begin{theorem}\label{rs}If $\{\nu_N\}$ is
$k$-pseudo-random, and $\{A_{N}\cap B_{\varepsilon N}\}$ has
positive upper density relative to $\{\nu_N\}$, then there is a
subset $A_N$ that contains a truncated principal residue class of
${\frak b}$ of size $|O_K\cap B_k|$.
\end{theorem}
\proof We have
$$\frac{\sum_{x\in A_{N}\cap
B_{\varepsilon N}}\tilde{\nu}_N(x)}{\sum_{x\in {\frak b}/(N{\frak
b})}\tilde{\nu}_N(x)}=\frac{1}{N^{[K:{\mathbb Q}]}}\sum_{x\in
A_{N}\cap B_{\varepsilon
N}}\tilde{\nu}_N(x)+o(1)$$$$=\frac{1}{N^{[K:{\mathbb Q}]}}\sum_{x\in
A_{N}\cap B_{\varepsilon N}}\nu_N(x)+o(1)=\frac{\sum_{x\in A_N\cap
B_{\varepsilon N}}\nu_N(g)}{\sum_{x\in {\frak b}\cap
B_N}\nu_N(x)}+o(1).$$ So $\{A_{N}\cap B_{\varepsilon N}({\rm mod}
N{\frak b})\}$
 has positive upper density
relative to $\{\tilde{\nu}_N\}$. By Theorem \ref{rsinversesys},
there is a subset $A_{N}\cap B_{\varepsilon N}({\rm mod} N{\frak
b})$, a truncated  residue class of ${\frak b}$ of size $|O_K\cap
B_k|$ such that
$$A({\rm mod} N{\frak b})\subseteq A_{N}\cap B_{\varepsilon N}({\rm mod} N{\frak b}).$$
As $A$ is bounded, and $\varepsilon$ is sufficiently small, we
conclude that
$$A\subseteq A_{N}\cap B_{\varepsilon N}.$$ The theorem follows.
\endproof

\section{The cross-correlation of the truncated von Mangolt function} In this section
we shall establish the cross-correlation of the truncated von
Mangolt function.

The truncated von Mangolt function for the rational number field was
introduced by Heath-Brown \cite{Heath-Brown}. The truncated von
Mangolt function for the Gaussian number field was introduced by Tao
\cite{tao:gaussian}. The cross-correlation of the truncated von
Mangolt function for the rational number field were studied by
Goldston-Y{\i}ld{\i}r{\i}m in \cite{gy1-cor, gy2-cor, gy3-cor}, and
by Green-Tao in \cite{gt-ap, gt-lattice}.

Let
$\varphi: {\mathbb R} \to {\mathbb R}^+$
be  a smooth bump function supported on $[-1,1]$ which equals 1 at 0, and let $R>1$ be a parameter.
 We now define the
truncated von Mangoldt function for the number field
$K$.\begin{definition}We define the truncated \emph{von Mangoldt
function} $\Lambda_{K,R}$ of $K$ by the formula
$$
\Lambda_{K,R}({\frak n}) := \sum_{{\frak d}|{\frak
n}}\mu_K({\frak d})\varphi (\frac{\log {\rm N}({\frak d})}{\log
R}),$$
where $\mu_K$ is the \emph{M\"obius function}
of $K$ defined by the formula
$$\mu_K({\frak n})=\left\{
                      \begin{array}{ll}
                        (-1)^k, & \hbox{} {\frak n}\text{ is a product of }k\text{ distinct prime ideals},\\
                        0, & \hbox{otherwise.}
                      \end{array}
                    \right.
$$
\end{definition}
Note that $\Lambda_{K,R}({\frak n})=1$ if ${\frak n}$ is a prime ideal
with norm $\geq R$.

Let $\zeta_K(z)$ be the zeta function of $K$,
$\phi_K(W):=|O_K/(W))^{\times}|$, $$ \hat{\varphi}(x)=
\int_{-\infty}^\infty e^t \varphi(t) e^{ixt}\ dt,$$and
$$ c_{\varphi} :=\int_{-\infty}^{+\infty}\int_{-\infty}^{+\infty} \frac{(1+iy)(1+iy')}{ (2+iy+iy')}\hat{\varphi}(y) \hat{\varphi}(y') dy
dy'.
$$
From now on, for each $N\in I$, let $$\nu_N(x)=\frac{\phi_K(W)\log
R\cdot{\rm Res}_{z=1}\zeta_K(z)}{ c_{\varphi} W^{[K:{\mathbb
Q}]}}\Lambda_{K,R}^2((Wx+\alpha){\frak b}^{-1}).$$  Here $$\log
R=\frac{\log N}{8|O_K\cap B_k|2^{|O_K\cap
 B_k|}},$$ $W$ is the
product of prime numbers $\leq w:=\log\log N$, and
 $\alpha$ a number prime to $W$

 We now prove the following.
\begin{theorem}\label{crosscorrelation}The system $\{\nu_N\}$ satisfies the $k$-cross-correlation condition.
\end{theorem}
\proof Given any parallelotope $I$ in $K_{\infty}$, given any
positive integers $s\leq |O_K\cap B_k|2^{|O_K\cap B_k|},m\leq
2|O_K\cap B_k|$, given any $N\log^{-2sA}N<\lambda<N$, and given any
mutually independent linear forms $\psi_1,\cdots,\psi_{s}$ in $m$
variables whose coefficients are numbers in $O_K\cap B_{2k}$, we
show that
$$\frac{1}{|{\frak b}\cap(\lambda I)|^m}\sum_{\stackrel{x_i\in {\frak b}\cap(\lambda I)}{i=1,\cdots,m}}\prod_{j=1}^{s}\nu_N(\psi_j(x)+b_j)=1+o(1),\
N\rightarrow\infty$$ uniformly for all numbers $b_1,\cdots,b_{s}\in
{\frak b}$.

We define
$${\frak
S}=\sum_{{\frak d},{\frak d}'}\omega(({\frak d}_i\cap{\frak
d}_i')_{1\leq i\leq s})\prod_{i=1}^s\mu_{K}({\frak d}_i)\mu_{K}({\frak
d}_i')\varphi(\frac{\log{\rm N}{\frak d}_i}{\log
R})\varphi(\frac{\log{\rm N}{\frak d}'_i}{\log R}),$$ where ${\frak
d}$ and ${\frak d}'$ run over $s$-tuples of ideals of $O_K$, and
$$\omega(({\frak d}_i)_{1\leq i\leq s})=\frac{|\{x\in({\frak b}/({\frak b}\cdot \cap_{i=1}^s{\frak d}_i))^m:{\frak d}_i|
(W\psi_i(x) + b'_i){\frak b}^{-1},\forall i=1,\cdots,s\}|}{({\rm
N}\cap_{i=1}^s{\frak d}_i)^{m}}$$ with $b'_i=Wb_i+\alpha$.

Let $\{\gamma_{j}\}$ ($j=1,\cdots,[K:{\mathbb Q}]$) be a ${\mathbb
Z}$-basis of ${\frak b}$ such that
$\{\lambda_{j}\gamma_{j}\}$ is a
${\mathbb Z}$-basis of ${\frak b}\cdot\cap_{i=1}^s{\frak d}_i$,
where each $\lambda_{i}$ is a positive integer. Set
$$I_0=\{x\in K_{\infty}:
x_i\in\sum_{j=1}^{[K:{\mathbb
Q}]}(0,1]\lambda_{j}\gamma_{j}\}.$$Then $$\omega(({\frak
d}_i)_{1\leq i\leq s}))=\frac{|\{x\in (I_0\cap{\frak b})^m:{\frak
d}_i| (W\psi_i(x) + b'_i){\frak b}^{-1},\forall
i=1,\cdots,s\}|}{({\rm N}\cap_{i=1}^s{\frak d}_i)^{m}}.$$ The number
of translates of $I_0^m$ by vectors in $({\frak
b}\cdot\cap_{i=1}^s{\frak d}_i)^m$ which intersect the boundary of
$\lambda I^m$ is bounded by $O(\lambda^{m[K:{\mathbb
Q}]-1}/(\prod_{j=1}^{[K:{\mathbb Q}]}\lambda_{j})^{m-1})$. So the
number of translates of $I_0^m$ by vectors in $({\frak
b}\cdot\cap_{i=1}^s{\frak d}_i)^m$ which lie in the interior of
$\lambda I^m$ is $$\frac{{\rm Vol} (I)^m}{{\rm
Vol}(I_0)^m}\lambda^{m[K:{\mathbb Q}]}+ O(\lambda^{m[K:{\mathbb
Q}]-1}{\rm N}(\cap_{i=1}^s{\frak d}_i)/(\prod_{j=1}^{[K:{\mathbb
Q}]}\lambda_{j})^{m}).$$ It follows that
$$ \frac{|\{x\in
(\lambda I\cap{\frak b})^m:{\frak d}_i| (W\psi_i(x) + b'_i){\frak
b}^{-1},\forall i=1,\cdots,s\}|}{\lambda^{m[K:{\mathbb Q}]} {\rm
Vol}(I)^m/(\sqrt{|d_K|}{\rm N}{\frak b})^{[K:{\mathbb Q}]}}=
\omega(({\frak d}_1)_{1\leq i\leq s}) + O(\frac{{\rm
N}(\cap_{i=1}^s{\frak d}_i)}{\lambda}),$$ where $d_K$ is the
discriminant of $K$. From that estimate one can
infer$$\frac{(\sqrt{|d_K|}{\rm N}{\frak
b})^{m}}{\lambda^{m[K:{\mathbb Q}]} {\rm Vol}(I)^m}\sum\limits_{ x
\in (\lambda I\cap{\frak b})^m}\prod_{i=1}^s
\Lambda_{K,R}^2((W\psi_i(x) + b'_i){\frak b}^{-1})={\frak
S}+O(\frac{R^{4s}}{\lambda}).
$$
Therefore we are is reduced to the following.
$${\frak
S}=(1 + o(1))(\frac{ c_{\varphi} W^{[K:{\mathbb Q}]}}{\phi_K(W)\log
R\cdot{\rm Res}_{z=1}\zeta_K(z)})^{s}.
$$
We define
$$F(t,t')=\sum_{{\frak d},{\frak d}'}\omega(({\frak
d}_j\cap{\frak d}'_j)_{1\leq j\leq s})\prod_{j=1}^s\frac{\mu_{K}({\frak
d}_j)\mu_{K}({\frak d}'_j)} {\N({\frak d}_j)^{\frac{1+it_j}{\log
R}}\N({\frak d}'_j)^{\frac{1+it'_j}{\log R}}},\ t,t'\in{\mathbb
R}^s,$$where ${\frak d}$ and ${\frak d}'$ run over $s$-tuples of
ideals of $O_K.$

 It is easy to see that, for all $B >
0$,
$$
e^x\varphi(x) = \int_{-\sqrt{\log R}}^{\sqrt{\log R}}
\hat{\varphi}(t) e^{-ixt}\ dt+O((\log R)^{-B}).
$$
It follows that for all $B>0$,$${\frak S}= \int_{[-\sqrt{\log
R},\sqrt{\log R}]^s}\int_{[-\sqrt{\log R},\sqrt{\log R}]^s}
F(t,t')\hat{\varphi}(t)\hat{\varphi}(t')dtdt'$$$$+O((\log
R)^{-B})\cdot\sum_{{\frak d},{\frak d}'}\omega(({\frak
d}_j\cap{\frak d}'_j)_{1\leq j\leq
s})\prod_{i=1}^s\frac{|\mu_{K}({\frak d}_j)\mu_{K}({\frak d}'_j)|}
{\N({\frak d}_j)^{1/\log R}\N({\frak d}'_j)^{1/\log R}}.
$$
Hence we are reduced to prove the following.
$$\sum_{{\frak d},{\frak d}'}\omega(({\frak
d}_j\cap{\frak d}'_j)_{1\leq j\leq s})\prod_{j=1}^s\frac{|\mu_{K}({\frak
d}_j)\mu_{K}({\frak d}'_j)|} {\N({\frak d}_j)^{1/\log R}\N({\frak
d}'_j)^{1/\log R}}\ll\log^{O_{s}(1)}R,$$ and, for
$t,t'\in[-\sqrt{\log R},\sqrt{\log
R}]^s$,
$$
F(t,t')=(1 +
o(1))(\frac{W^{[K:{\mathbb
Q}]}}{\phi_K(W)\log R\cdot{\rm Res}_{z=1}\zeta_K(z)})^{s}\prod_{j=1}^s\frac{(1+it_j)(1+it'_j)}{ (2+it_j+it'_j)}.$$

We prove the equality first. Applying the Chinese remainder theorem,
one can show that
$$\omega(({\frak
d}_j)_{1\leq j\leq s})=\prod_{\wp}\omega(({\frak d}_j,\wp)_{1\leq j\leq s}),$$
where $\wp$ runs over nonzero prime ideals of $O_K$. One
can also show that
$$\omega((({\frak
d}_j,\wp))_{1\leq j\leq s})=\left\{
                     \begin{array}{ll}1, & \hbox{ } \prod_{j=1}^s({\frak
d}_j,\wp)=(1),\\
0,&\hbox{}  \prod_{j=1}^s({\frak d}_j,\wp)\neq(1), \wp|W.
                     \end{array}
                   \right.
$$
And, if $\wp\nmid W$ and $W$ is sufficiently large,
then one can show that
$$\omega((({\frak
d}_j,\wp))_{1\leq j\leq s})\left\{
                    \begin{array}{ll}
                      =1/{\rm N}\wp, & \hbox{} \prod_{j=1}^s({\frak d}_j,\wp)=\wp\\
                      \leq1/{\rm N}\wp^2, & \hbox{}\wp^2\mid \prod_{j=1}^s({\frak
d}_j,\wp).
                    \end{array}
                  \right.
$$
It follows that
$$
F(t,t')=\prod_{\wp}\sum_{{\frak d}_j,{\frak d}'_j\mid
\wp,\forall j=1,\cdots,s}\omega(({\frak d}_j\cap{\frak d}'_j)_{1\leq j\leq s})\prod_{j=1}^s\frac{\mu_{K}({\frak d}_j)\mu_{K}({\frak d}'_j)}
{{\rm N}{\frak d}_j^{\frac{1+it_j}{\log R}}{\rm N}{{\frak
d}'_j}^{\frac{1+it'_j}{\log R}}}
$$$$ =\prod_{\wp\nmid W}(1 + \sum_{j=1}^s-{\rm
N}\wp^{-1-\frac{1+it_j}{\log R}} -{\rm N}\wp^{-1-\frac{1+it'_j}{\log
R}}+ {\rm N}\wp^{-1-\frac{2+it_j+it'_j}{\log R}}+
O_{s}(\frac{1}{{\rm N}\wp^2}))$$$$ =\prod_{p\nmid W}(1 +
O_{s}(\frac{1}{p^2})) \prod_{j=1}^s\prod_{\wp\nmid W}\frac{(1
-{\rm N}\wp^{-1-\frac{1+it_j}{\log R}})(1 -{\rm
N}\wp^{-1-\frac{1+it'_j}{\log R}})}{ (1-{\rm
N}\wp^{-1-\frac{2+it_j+it'_j}{\log R}})}$$
$$ =(1+O(\frac{1}{\log R})) \prod_{j=1}^s \frac{
\zeta_K(1+\frac{2+it_s+it'_s}{\log R})}{\zeta_K(1+\frac{1+it_s}{\log
R})\zeta_K(1+\frac{1+it'_j}{\log R})}\prod_{\wp\mid W}\frac{ (1-{\rm
N}\wp^{-1-\frac{2+it_j+it'_j}{\log R}})}{(1 -{\rm
N}\wp^{-1-\frac{1+it_j}{\log R}})(1 -{\rm
N}\wp^{-1-\frac{1+it'_j}{\log R}})}.$$ From the estimate
$$\zeta_K(z)=\frac{{\rm Res}_{z=1}\zeta_K(z)}{z-1}+O(1), \ z\to 1,$$
and the estimate $$e^z=1+O(z),\ z\to 0,$$ we infer that
$$ F(t,t')
=(1 + O(\frac {1}{\log R}))\cdot\prod_{\wp\mid W}(1+O(\frac{\log
{\rm N}\wp}{{\rm N}\wp\log^{1/2} R}))\cdot$$$$(\frac{W^{[K:{\mathbb
Q}]}}{\phi_K(W)\log R\cdot{\rm Res}_{z=1}\zeta_K(z)})^{s}\prod_{j=1}^s\frac{(1+it_j)(1+it'_j)}{ (2+it_j+it'_j)}.$$ Applying the
estimate
$$\prod_{\wp\mid W}(1+\frac{\log {\rm
N}\wp}{{\rm N}\wp})=O(e^{\log^2w}),$$ we arrive at
$$ F(t,t')
=(1 + o(1))(\frac{W^{[K:{\mathbb
Q}]}}{\phi_K(W)\log R\cdot{\rm Res}_{z=1}\zeta_K(z)})^{s}\prod_{j=1}^s\frac{(1+it_j)(1+it'_j)}{ (2+it_j+it'_j)}$$ as required.

We now turn to prove the estimate $$\sum_{{\frak d},{\frak
d}'}\omega(({\frak d}_j\cap{\frak d}'_j)_{1\leq j\leq s})\prod_{j=1}^s\frac{|\mu_{K}({\frak d}_j)\mu_{K}({\frak d}'_j)|} {\N({\frak
d}_j)^{1/\log R}\N({\frak d}'_j)^{1/\log R}}\ll\log^{O_{s}(1)}R.$$
We have $$\sum_{{\frak d},{\frak d}'}\omega(({\frak d}_j\cap{\frak
d}'_j)_{1\leq j\leq s})\prod_{j=1}^s\frac{|\mu_{K}({\frak
d}_j)\mu_{K}({\frak d}'_j)|} {\N({\frak d}_j)^{1/\log R}\N({\frak
d}'_j)^{1/\log R}}$$$$=\prod_{\wp}\sum_{{\frak d}_j,{\frak
d}'_j\mid \wp,\forall j=1,\cdots,s}\omega(({\frak d}_j\cap{\frak
d}'_j)_{1\leq j\leq s})\prod_{j=1}^s\frac{1} {{\rm N}{\frak
d}_j^{\frac{1}{\log R}}{\rm N}{{\frak d}'_j}^{\frac{1}{\log R}}}$$$$
=\prod_{\wp\nmid W}(1 + {\rm N}\wp^{-1-\frac{1}{\log
R}})^{O(1)}$$$$ =\prod_{p}(1 + p^{-1-\frac{1}{\log
R}})^{O(1)}=\zeta(1+\frac{1}{\log R})^{O(1)}\ll\log^{O(1)}R.$$
 This completes the proof of the theorem.\endproof

\section{The auto-correlation of the truncated von Mangolt function} In this section
we shall establish the auto-correlation of the truncated von Mangolt
function.

The auto-correlation of the truncated von Mangolt function for the
rational number field was studied by Goldston-Y{\i}ld{\i}r{\i}m in
\cite{gy1-cor, gy2-cor, gy3-cor}, and by Green-Tao in \cite{gt-ap,
gt-lattice}.

 We now prove the following.
\begin{theorem}\label{autocorrelation}The system $\{\nu_N\}$ satisfies the $k$-auto-correlation condition.
\end{theorem}
The above theorem follows from the following lemma.
\begin{lemma}Let $I$ be any parallelotope in $K_{\infty}$. Then
$$\frac{1}{|(NI)\cap{\frak b}|}
\sum\limits_{x\in (NI)\cap{\frak b}}\prod_{i=1}^s
 \nu_N(x +
y_i) \ll \prod_{1\leq i<j\leq s}\prod_{\wp\mid(y_i-y_j)}(1+
O_{s}(\frac{1}{{\rm N}\wp}))$$ uniformly for all $s$-tuples
$y\in{\frak b}^s$ with distinct coordinates.\end{lemma}

\proof We define
$${\frak
S}_2=\sum_{{\frak d},{\frak d}'}\omega_2(({\frak d}_i\cap{\frak
d}'_i)_{1\leq i\leq s})\prod_{i=1}^s\mu_{K}({\frak
d}_i)\mu_{K}({\frak d}'_i)\varphi(\frac{\log{\rm N}{\frak d}_i}{\log
R})\varphi(\frac{\log{\rm N}{\frak d}'_i}{\log R}),$$ where ${\frak
d}$ and ${\frak d}'$ run over $s$-tuples of ideals of $O_K$, and
$$\omega_2(({\frak d}_i)_{1\leq i\leq s})
=\frac{|\{x\in{\frak b}/({\frak b}\cdot\cap_{i=1}^s{\frak d}_i):
{\frak d}_s| (Wx + h_i){\frak b}^{-1},\forall i=1,\cdots,s\}|}{({\rm
N}\cap_{i=1}^s{\frak d}_i)},$$ where $h_i=Wb(y)+Wy_i+\alpha$.

Let $\{\gamma_{j}\}$ ($j=1,\cdots,[K:{\mathbb Q}]$) be a ${\mathbb
Z}$-basis of ${\frak b}$ such that $\{\lambda_{j}\gamma_{j}\}$ is a
${\mathbb Z}$-basis of ${\frak b}\cdot\cap_{i=1}^s{\frak d}_i$,
where each $\lambda_{i}$ is a positive integer. Set
$$I_0=\{x\in K_{\infty}:
x_i\in\sum_{j=1}^{[K:{\mathbb
Q}]}(0,1]\lambda_{j}\gamma_{j}\}.$$Then $$\omega_2(({\frak
d}_i)_{1\leq i\leq s}))=\frac{|\{x\in I_0\cap{\frak b}:{\frak d}_i|
(Wx + h_i){\frak b}^{-1},\forall i=1,\cdots,s\}|}{({\rm
N}\cap_{i=1}^s{\frak d}_i)}.$$

The number of translates of $I_0$ by vectors in ${\frak
b}\cdot\cap_{i=1}^s{\frak d}_i$ which intersect the boundary of
$\lambda I$ is bounded by $O(\lambda^{[K:{\mathbb Q}]-1})$. So the
number of translates of $I_0$ by vectors in ${\frak
b}\cdot\cap_{i=1}^s{\frak d}_i$ which lie in the interior of
$\lambda I$ is
$$\frac{{\rm Vol} (I)}{{\rm Vol}(I_0)}\lambda^{[K:{\mathbb Q}]}+
O(\lambda^{[K:{\mathbb Q}]-1}/{\rm N}(\cap_{i=1}^s{\frak
d}_i)^{[K:{\mathbb Q}]-1}).$$ It follows that
$$ \frac{{\rm N}{\frak a}\sqrt{|d_K|}|\{x\in
\lambda I\cap{\frak b}:{\frak d}_i| (Wx + h_i){\frak b}^{-1},\
\forall i=1,\cdots,s\}|}{\lambda^{[K:{\mathbb Q}]} {\rm
Vol}(I)}$$$$= \omega_2(({\frak d}_i)_{1\leq i\leq s}) + O(\frac{{\rm
N}(\cap_{s\in S}{\frak d}_s)}{\lambda}).$$ From that estimate one
can infer$$(\frac{W^{[K:{\mathbb Q}]}}{\phi_K(W)\log
R})^{s}\frac{1}{|(NI)\cap{\frak b}|} \sum\limits_{x\in
(NI)\cap{\frak b}}\prod_{i=1}^s
 \nu_N(x +
y_i)={\frak S}_2+O(\frac{R^{4s}}{\lambda}).
$$
So we are reduced to proving that
$${\frak
S}_2 \ll (\frac{W^{[K:{\mathbb Q}]}}{\phi_K(W)\log
R})^{s}\prod_{\wp\mid\Delta}(1+ O_{s}(\frac{1}{{\rm N}\wp})),$$
whenever
$$\Delta:=\prod_{i\neq j}(y_i-y_j)\neq0.$$
 We define
$$F_2(t,t')=\sum_{{\frak d},{\frak d}'}\omega_2(({\frak
d}_i\cap{\frak d}'_i)_{1\leq i\leq
s})\prod_{j=1}^s\frac{\mu_{K}({\frak d}_j)\mu_{K}({\frak d}'_j)}
{\N({\frak d}_j)^{\frac{1+it_j}{\log R}}\N({\frak
d}'_j)^{\frac{1+it'_j}{\log R}}},\ t,t'\in{\mathbb R}^s,$$where
${\frak d}$ and ${\frak d}'$ run over $s$-tuples of ideals of $O_K.$

For all $B>0$, we have
$${\frak S}_2=\int_{[-\sqrt{\log R},\sqrt{\log R}]^s}\int_{[-\sqrt{\log R},\sqrt{\log R}]^s}
F_2(t,t')\psi(t)\psi(t')dtdt'$$$$+O_B((\log
R)^{-B})\cdot\sum_{{\frak d},{\frak d}'}\omega_2(({\frak
d}_i\cap{\frak d}'_i)_{1\leq i\leq
s})\prod_{j=1}^s\frac{|\mu_{K}({\frak d}_j)\mu_{K}({\frak d}'_j)|}
{\N({\frak d}_j)^{1/\log R}\N({\frak d}'_j)^{1/\log R}}.
$$
Hence we are reduced to prove the following.
$$\sum_{{\frak d},{\frak d}'}\omega_2(({\frak
d}_i\cap{\frak d}'_i)_{1\leq i\leq
s})\prod_{j=1}^s\frac{|\mu_{K}({\frak d}_j)\mu_{K}({\frak d}'_j)|}
{\N({\frak d}_j)^{1/\log R}\N({\frak d}'_j)^{1/\log
R}}\ll\log^{O_{s}(1)}R,$$ and, for $t,t'\in[-\sqrt{\log
R},\sqrt{\log R}]^s$,
$$
F_2(t,t') \ll (\frac{W^{[K:{\mathbb Q}]}}{\phi_K(W)\log
R})^{s}\prod_{\wp\mid\Delta,\wp\nmid W}(1+ O_{s}(\frac{1}{{\rm
N}\wp}))\prod_{j=1}^s\frac{(1+|t_j|)(1+|t'_j|)}{
(2+|t_j|+|t'_j|)}.$$

We prove the second inequality but omit the proof of first one.
Applying the Chinese remainder theorem, one can show that
$$\omega_2(({\frak
d}_i)_{1\leq i\leq s})=\prod_{\wp}\omega_2(({\frak d}_i,\wp)_{1\leq
i\leq s}).$$ One can also show that
$$\omega_2((({\frak
d}_i,\wp))_{1\leq i\leq s})=\left\{
                     \begin{array}{ll}1, & \hbox{ } \prod_{i=1}^s({\frak
d}_i,\wp)=(1),\\
0,&\hbox{}  \prod_{i=1}^s({\frak d}_i,\wp)\neq(1), \wp|W.
                     \end{array}
                   \right.
$$
And, if $\wp\nmid W$ and $w$ is sufficiently large, then one can
show that
$$\omega_2((({\frak
d}_i,\wp))_{1\leq i\leq s})\left\{
                    \begin{array}{ll}
                      =1/{\rm N}\wp, & \hbox{} \prod_{i=1}^s({\frak d}_i,\wp)=\wp\\
                     =0, & \hbox{}\wp^2\mid \prod_{i=1}^s({\frak
d}_i,\wp),\wp\nmid\Delta,\\
\leq1/{\rm N}\wp, & \hbox{}\wp^2\mid \prod_{s\in S}({\frak
d}_s,\wp),\wp\mid\Delta.
                    \end{array}
                  \right.
$$
It follows that
$$
F_2(t,t')=\prod_{\wp}\sum_{{\frak d}_i,{\frak d}'_i\mid \wp,\forall
i=1,\cdots,s}\omega_2(({\frak d}_i\cap{\frak d}'_i)_{1\leq i\leq
s})\prod_{j=1}^s\frac{\mu_{K}({\frak d}_j)\mu_{K}({\frak d}'_j)}
{{\rm N}{\frak d}_s^{\frac{1+it_j}{\log R}}{\rm N}{{\frak
d}'_j}^{\frac{1+it'_j}{\log R}}}
$$$$ =\prod_{\wp\nmid W\Delta}(1 + \sum_{j=1}^s-{\rm
N}\wp^{-1-\frac{1+it_j}{\log R}} -{\rm N}\wp^{-1-\frac{1+it'_j}{\log
R}}+ {\rm N}\wp^{-1-\frac{2+it_j+it'_j}{\log R}})\prod_{\wp\nmid
W,\wp\mid\Delta}(1+ O_{s}(\frac{1}{{\rm N}\wp}))$$$$\ll
\prod_{\wp\mid\Delta,\wp\nmid W}(1+ O_{s}(\frac{1}{{\rm
N}\wp}))\prod_{j=1}^s\prod_{\wp\nmid W\Delta}\frac{(1 -{\rm
N}\wp^{-1-\frac{1+it_j}{\log R}})(1 -{\rm
N}\wp^{-1-\frac{1+it'_j}{\log R}})}{ (1-{\rm
N}\wp^{-1-\frac{2+it_j+it'_j}{\log R}})}$$
$$ \ll (\frac{W^{[K:{\mathbb Q}]}}{\phi_K(W)})^{s}\prod_{\wp\mid\Delta,\wp\nmid W}(1+
O_{s}(\frac{1}{{\rm N}\wp})) \prod_{j=1}^s \frac{
\zeta_K(1+\frac{2+it_j+it'_j}{\log R})}{\zeta_K(1+\frac{1+it_j}{\log
R})\zeta_K(1+\frac{1+it'_j}{\log R})}$$$$
 \ll (\frac{W^{[K:{\mathbb Q}]}}{\phi_K(W)\log R})^{s}\prod_{\wp\mid\Delta,\wp\nmid W}(1+
O_{s}(\frac{1}{{\rm N}\wp}))\prod_{j=1}^s\frac{(1+|t_j|)(1+|t'_j|)}{
(2+|t_j|+|t'_j|)}.$$
 This completes the proof of the lemma.\endproof
\section{Proof of the main theorem}
In this section we prove Theorem \ref{main}.

For each $N\in I$, and for each $\alpha\in{\frak b}$ with
$(\alpha,W{\frak b})={\frak b}$, set
$$A_{N,\alpha}=\{x\in{\frak b}\cap B_N\mid (Wx+\alpha){\frak b}\text{ is prime }\}.$$
By Theorem \ref{rs}, Theorem \ref{main} follows from the following
theorem.
\begin{theorem}For each $N\in I$, there is a number $\alpha_N\in(WG)\cap{\frak b}$
 with $(\alpha_N,W{\frak
b})={\frak b}$ such that the system $|\{A_{N,\alpha_N}\cap
B_{\varepsilon N}\}$ has positive upper density relative to
$\{\nu_N\}$.\end{theorem}
\proof Let $S_{K,\infty}$ the set of
infinite places of $K$. One can prove that there is a positive
constant $c_K$ such that every principal fractional ideal of $K$ has
a generator $\xi$ satisfying
\[|\sigma(\xi)|\leq c_K({\rm N}(\xi))^{1/[K:{\mathbb Q}]},\ \forall \sigma\in S_{K,\infty}.\]
It follows that, for each $N\in I$,
and for any prime ideal $\wp\in[{\frak b}^{-1}]$ satisfying
$(\wp,W)=1 $ and ${\rm N}\wp\leq c_K^{-1}{\rm N}{\frak b}^{-1}\cdot(
NW\varepsilon/2)^{[K:{\mathbb Q}]}$, there is a number
$\alpha\in{\frak b}\cap(WG)$ with $(\alpha,W{\frak b})={\frak b}$,
and a number $x\in A_{N,\alpha}\cap B_{\varepsilon N}$ such that
$\wp=(Wx+\alpha){\frak b}^{-1}$. So
$$\sum_{\stackrel{(\alpha,W{\frak b})={\frak
b}}{\alpha\in{\frak b}\cap(WG)}}\sum_{x\in A_{N,\alpha}\cap
B_{\varepsilon N}}\Lambda_{K,R}^2((Wx+\alpha){\frak
b}^{-1})$$$$\geq\sum_{\stackrel{\wp\in[{\frak
b}^{-1}],(\wp,W)=1}{c/2<{\rm N}\wp\cdot(NW)^{-[K:{\mathbb Q}]}\leq
c}}\Lambda_{K,R}^2(\wp)\gg(NW)^{[K:{\mathbb Q}]}/\log N ,$$ where
$c=c_K^{-1}{\rm N}{\frak b}^{-1}\cdot(\varepsilon/2)^{[K:{\mathbb
Q}]}$. The theorem now follows by the pigeonhole principle.
\endproof

\end{document}